\documentclass[12pt,amscd]{amsart}
\usepackage[all]{xy}
\usepackage{graphicx,tikz}
\usepackage{amsmath,amsxtra,amssymb,latexsym, amscd,amsthm}
\usepackage{indentfirst}
\usepackage[mathscr]{eucal}
\usepackage[pagebackref=true]{hyperref}

\newtheorem{thm}{Theorem}[section]
\newtheorem{cor}[thm]{Corollary}
\newtheorem{lem}[thm]{Lemma}

\theoremstyle{definition}
\newtheorem{defn}[thm]{Definition}

\newtheorem{rem}[thm]{Remark}

\numberwithin{equation}{section}

\DeclareMathOperator{\dstab}{dstab}

\DeclareMathOperator{\depth}{depth}
\DeclareMathOperator{\pd}{pd}
\DeclareMathOperator{\supp}{supp}

\def\B {\mathcal B}

\def\e {\mathbf e}

\def\m {\mathfrak m}

\def\k {\mathrm{k}}

\begin{document}

\title{Depth of powers of edge ideals of Cohen-Macaulay trees}

\author{Nguyen Thu Hang}
\address{Thai Nguyen University of Sciences, Tan Thinh Ward, Thai Nguyen City, Thai Nguyen, Vietnam}
\email{hangnt@tnus.edu.vn}

\author{Truong Thi Hien}
\address{Hong Duc University, 565 Quang Trung Street, Dong Ve Ward, Thanh Hoa, Vietnam}
\email{hientruong86@gmail.com}

\author{Thanh Vu}
\address{Institute of Mathematics, VAST, 18 Hoang Quoc Viet, Hanoi, Vietnam}
\email{vuqthanh@gmail.com}

\subjclass[2020]{05E40, 13D02, 13F55}
\keywords{depth of powers; Cohen-Macaulay trees; Cohen-Macaulay bipartite graphs}

%\subjclass{13D45, 05C90, 05E40, 05E45.}
%\keywords{Matroid, arboricity, Stanley-Reisner ideal, regularity.}
\date{}

\dedicatory{Dedicated to Professor Ngo Viet Trung on the occasion of his 70th birthday}
\commby{}
%-----------------------------------------------------------
% -----------------------------------------------------------
\maketitle
% -----------------------------------------------------------
\begin{abstract}
    Let $I$ be the edge ideal of a Cohen-Macaulay tree of dimension $d$ over a polynomial ring $S = \k[x_1,\ldots,x_{d},y_1,\ldots,y_d]$. We prove that for all $t \ge 1$, 
    $$\depth (S/I^t) =  \max \{d -t + 1, 1 \}.$$
\end{abstract}

\maketitle

\section{Introduction}
\label{sect_intro}
Let $S = \k[x_1, \ldots, x_n]$ be a standard graded polynomial ring over a field $\k$. For a homogeneous ideal $I \subset S$, we denote by $\dstab(I)$ the {\it index of depth stability} of $I$, i.e., the smallest positive natural number $k$ such that $\depth S/I^\ell = \depth S/I^k$ for all $\ell \ge k$. Such a number exists due to the result of Brodmann \cite{Br}. Let $G$ be a simple graph on the vertex set $V(G) = \{x_1,\ldots,x_n\}$ and edge set $E(G) \subseteq V(G) \times V(G)$. The edge ideal of $G$, denoted by $I(G)$, is the squarefree monomial ideal generated by $x_i x_j$ where $\{x_i,x_j\}$ is an edge of $G$. In a fundamental paper \cite{T}, Trung found a combinatorial formula for $\dstab(I(G))$ for large classes of graphs, including unicyclic graphs. In particular, when $G$ is a tree, $\dstab(I(G)) = n - \epsilon_0(G)$ where $\epsilon_0(G)$ is the number of leaves of $G$. Though we know the limit depth and its index of depth stability, intermediate values for depth of powers of edge ideals were only known in very few cases, e.g., for path graphs by Balanescu and Cimpoeas \cite{BC} and cycles and starlike trees by Minh, Trung, and the last author \cite{MTV1}. While the depth of powers of edge ideals of general trees is very mysterious \cite{MTV1}, in this paper, we compute the depth of powers of edge ideals of all Cohen-Macaulay trees.

In \cite{V}, Villarreal classified all Cohen-Macaulay trees. It says that a tree $G$ is Cohen-Macaulay if and only if it is the whisker graph of another tree $T$. In other words, $V(G) = V(T) \cup \{y_1, \ldots, y_d\}$ and $E(G) = E(T) \cup \{\{x_i,y_i\} \mid i = 1, \ldots, d\}$, where $T$ is an arbitrary tree on $d$ vertices. While the structure of $T$ could be very complicated, surprisingly, the depth of powers of $G$ does not depend on $T$. Namely,

\begin{thm}\label{depth_power_CM_tree}
    Let $I(G) \subset S = \k[x_1, \ldots,x_{d}, y_1, \ldots, y_d]$ be the edge ideal of a Cohen-Macaulay tree $G$ of dimension $d$. Then for all $t \ge 1$,
    $$\depth S/I(G)^t =  \max \{d-t+1, 1\}.$$
\end{thm}

We now outline the ideas to carry out this computation. First, we show a general upper bound for the depth of powers of trees. For that purpose, we introduce some notation. Let $\e = \{e_1, \ldots, e_t\}$ be a set of $t$ distinct edges of $G$. We consider $\e$ itself as a subgraph of $G$ with edge set $E(\e) = \{e_1, \ldots, e_t\}$. We denote by $N[\e]$ the closed neighbourhood of $\e$ in $G$ (see Subsection \ref{subsection_graphs} for the definition of closed neighbourhood). Furthermore, $G[\e]$ denotes the induced subgraph of $G$ on $N[\e]$ and $G[\bar {\e}]$ denotes the induced subgraph of $G$ on $V(G) \setminus N[\e]$.

\begin{lem}\label{lem_a3} Let $\e = \{e_1, \ldots, e_t\}$ be a set of $t$ distinct edges of a simple graph $G$. Assume that $G[\e]$ is bipartite, $G[\bar {\e}]$ is weakly chordal, and $\e$, viewed as a subgraph of $G$, is connected. Let $R$ be the polynomial ring over the vertex set of $G[\bar{\e}]$. Then 
$$\depth S/I(G)^{t+1}  \le 1 + \depth R/I(G[\bar {\e}]).$$    
\end{lem}

For the lower bound, we prove the following general bound for the depth of powers of edge ideals of whisker trees.

\begin{lem}\label{lem_a4} Let $T$ be a forest on $n$ vertices. Let $U \subset V(T)$ be a subset of vertices of $T$, and $G$ be the new forest obtained by adding a whisker to each vertex in $U$. Namely, $V(G) = V(T) \cup \{y_i \mid x_i \in U\}$ and $E(G) = E(T) \cup \{ \{x_i, y_i\} \mid x_i \in U\}$. Then for all $t \ge 1$, we have 
$$\depth (R/I(G)^t) \ge |U| - t +  1,$$
where $R = \k[x_1,\ldots,x_n,y_j \mid x_j \in U]$ is the polynomial ring over the variables corresponding to vertices of $G$.
\end{lem}

Our method allows us to compute the depth of powers of the edge ideal of a Cohen-Macaulay bipartite graph constructed by Banerjee and Mukundan \cite{BM}.
\begin{thm}\label{thm_depth_BM_ideals} Let $S = k[x_1,\ldots,x_d,y_1,\ldots,y_d]$. For a fix integer $j$ such that $1 \le j \le d$, let $G_{d,j}$ be a bipartite graph whose edge ideal is $I(G_{d,j}) = (x_i y_i, x_1 y_i, x_k y_j \mid 1 \le i \le d, 1 \le k \le j)$. Then 
$$\depth S/I(G_{d,j})^s = \max(1,d - j -s + 3),$$
for all $s \ge 2$.    
\end{thm}

We structure the paper as follows. In Section \ref{sec_pre}, we set up the notation and provide some background. In Section \ref{sec_depth_power_CM_trees}, we prove Theorem \ref{depth_power_CM_tree}. In Section \ref{sec_depth_power_CM_bipartite_graph}, we prove Theorem \ref{thm_depth_BM_ideals}.

\section{Preliminaries}\label{sec_pre}
In this section, we recall some definitions and properties concerning the depth of monomial ideals and edge ideals of graphs. The interested readers are referred to \cite{BH, D} for more details.

Throughout this section, we denote $S = \k[x_1,\ldots, x_n]$ a standard graded polynomial ring over a field $\k$. Let $\m = (x_1,\ldots, x_n)$ be the maximal homogeneous ideal of $S$.

\subsection{Depth} For a finitely generated graded $S$-module $L$, the depth of $L$ is defined to be
$$\depth(L) = \min\{i \mid H_{\m}^i(L) \ne 0\},$$
where $H^{i}_{\m}(L)$ denotes the $i$-th local cohomology module of $L$ with respect to $\m$. 

\begin{defn} A finitely generated graded $S$-module $L$ is called Cohen-Macaulay if $\depth L = \dim L$. A homogeneous ideal $I \subseteq S$ is said to be Cohen-Macaulay if $S/I$ is a Cohen-Macaulay $S$-module.
\end{defn}

The following two results about the depth of monomial ideals will be used frequently in the sequence. The first one is \cite[Corollary 1.3]{R}. The second one is \cite[Theorem 4.3]{CHHKTT}.

\begin{lem}\label{lem_upperbound} Let $I$ be a monomial ideal and $f$ a monomial such that $f \notin I$. Then
    $$\depth S/I \le \depth S/(I:f)$$
\end{lem}

\begin{lem}\label{lem_depth_colon} Let $I$ be a monomial ideal and $f$ a monomial. Then 
$$\depth S/I \in \{\depth (S/I:f), \depth (S/(I,f))\}.$$    
\end{lem}

In the ideals of the form $I+(f)$ and $I:f$, some variables will be part of the minimal generators, and some will not appear in any of the minimal generators. A variable that does not divide any minimal generators of a monomial ideal $J$ will be called a free variable of $J$. We have 
\begin{lem}\label{lem_variables_clearing} Assume that $I = J + (x_a, \ldots, x_b)$ and $x_{b+1}, \ldots, x_n$ are free variables of $I$ where $J$ is a monomial ideal in $R = \k[x_1,\ldots,x_{a-1}]$. Then
$$\depth S/I = \depth R/J + (n-b).$$    
\end{lem}

\subsection{Graphs and their edge ideals}\label{subsection_graphs} 

Let $G$ denote a finite simple graph over the vertex set $V(G) = \{x_1,\ldots,x_n\}$ and the edge set $E(G)$. For a vertex $x\in V(G)$, let the neighbours of $x$ be the subset $N_G(x)=\{y\in V(G) \mid \{x,y\}\in E(G)\}$. The closed neighbourhood of $x$ is $N_G[x]=N_G(x)\cup\{x\}$. A vertex $x$ is called a leaf if it has a unique neighbour. An edge that contains a leaf is called a leaf edge. For a subset $U \subset V(G)$, $N_G(U) = \cup (N_G(x) \mid x \in U)$ and $N_G[U] = \cup (N_G[x] \mid x \in U)$. When it is clear from the context, we drop the subscript $G$ from the notation $N_G$.

Let $\e = \{e_1, \ldots, e_t\}$ be a set of $t$ distinct edges of $G$. We denote by $N[\e]$ the closed neighbourhood of $\e$ in $G$,
$$N[\e] = \cup (N[x] \mid x \text{ is a vertex of } e_j \text{ for some } j = 1, \ldots, t).$$
Furthermore, $G[\e]$ denotes the induced subgraph of $G$ on $N[\e]$ and $G[\bar {\e}]$ denotes the induced subgraph of $G$ on $V(G) \setminus N[\e]$.

A graph $H$ is called a subgraph of $G$ if $V(H) \subseteq V(G)$ and $E(H) \subseteq E(G)$. 

Let $U \subset V(G)$ be a subset of vertices of $G$. The induced subgraph of $G$ on $U$, denoted by $G[U]$, is the graph such that $V(G[U]) = U$ and for any vertices $u,v \in U$, $\{u,v\} \in E(G[U])$ if and only if $\{u,v\} \in E(G)$.

A set $\e = \{e_1, \ldots, e_t\}$ of $t$ distinct edges of $G$ is an induced matching if $e_i \cap e_j = \emptyset$ for all $i \neq j \in \{1, \ldots, t\}$ and $\e$ is an induced subgraph of $G$.

A tree is a connected acyclic graph. A cycle of length $m$ in $G$ is a sequence of distinct vertices $x_1,\ldots,x_m$ such that $\{x_1,x_2\}, \ldots, \{x_{m-1}x_m\} , \{x_m,x_1\}$ are edges of $G$.

A graph $G$ is bipartite if its vertex set has a decomposition $V(G) = U \cup V$ such that $E(G) \subset U \times V$. It is called a complete bipartite graph if $E(G) = U \times V$, denoted by $K_{U, V}$.

A graph $G$ is called weakly chordal if $G$ and its complement do not contain an induced cycle of length at least $5$.

The edge ideal of $G$ is defined to be
$$I(G)=(x_ix_j \mid \{x_i,x_j\}\in E(G))\subseteq S.$$
A graph $G$ is called Cohen-Macaulay if $I(G)$ is Cohen-Macaulay. For simplicity, we often write $x_i \in G$ (resp. $x_ix_j \in G$) instead of $x_i \in V(G)$ (resp. $\{x_i,x_j\} \in E(G)$). By abuse of notation, we also call $x_i x_j \in I(G)$ an edge of $G$.

We have the following result \cite[Lemma 2.10]{Mo}.

\begin{lem}\label{lem_colon_leaf} Suppose that $G$ is a graph and $xy$ is a leaf edge of $G$. Then for all $t \ge 2$, we have
$$I(G)^t : (xy) = I(G)^{t-1}.$$    
\end{lem}
As a consequence, we have the following well-known result.

\begin{lem}\label{lem_depth_power_non_increasing} Let $G$ be a graph. Assume that $G$ has a leaf edge. Then the sequence $\depth S/I(G)^t$ is non-increasing.    
\end{lem}
\begin{proof}
    Let $xy$ be a leaf edge of $G$. By Lemma \ref{lem_colon_leaf} and Lemma \ref{lem_upperbound}, we have 
    $$\depth S/I(G)^t \le \depth S/(I(G)^t : xy) = \depth S/I(G)^{t-1},$$
    for all $t \ge 2$. The conclusion follows.
\end{proof}

\subsection{Bipartite completion and a colon ideal} Let $H$ be a connected bipartite graph with the partition $V(H) = U \cup V$. The bipartite completion of $H$, denoted by $\widetilde {H}$, is the complete bipartite graph $K_{U, V}$. We have

\begin{lem}\label{lem_colon_product_edges} Let $\e = \{e_1, \ldots, e_t\}$ be a set of $t$ distinct edges of $G$. Assume that $\e$, viewed as a subgraph of $G$, is connected and $G[\e]$ is bipartite. Then 
$$I(G)^{t+1}:(e_1 \cdots e_t) = I(H),$$
where $H$ is a graph obtained from $G$ by bipartite completing $G[\e]$, i.e., $V(H) = V(G)$ and $E(H) = E(G) \cup E(\widetilde {G[\e]})$.    
\end{lem}
\begin{proof} We prove by induction on $t$. The base case $t = 1$ follows from \cite[Theorem 6.7]{B}. Now, assume that the statement holds for $t-1$. There always exists an edge $e_j \in \{e_1, \ldots, e_t\}$ so that $\e \setminus e_j$ is connected. We may assume that $j = t$. Let $\e' = \{e_1, \ldots, e_{t-1}\}$. By induction, we have 
\begin{equation}
    I(G)^t : (e_1 \cdots e_{t-1}) = I(H'),
\end{equation}
where $E(H') = E(G) \cup E(\widetilde{G[\e']})$.

By \cite[Theorem 6.7]{B}, $I(G)^{t+1} : (e_1 \cdots e_{t}) \supseteq I(H')^2 : e_t$ (see also \cite[Lemma 2.9]{MTV1}). Let $N_G[\e] = U \cup V$ be the partition of $N_G[\e]$. Assume that $e_t = uv$ with $u\in U$ and $v\in V$. Since $\e$ is connected, $\supp \e' \cap \{u,v\} \neq \emptyset$. We may assume that $u \in \supp \e'$. Hence,
\begin{equation}
    v \in N_G[\e'] \text{ and } N_G[\e] = N_G[\e'] \cup N_G(v).
\end{equation} 
In particular, $N_G[\e'] = U' \cup V$ and $U = U' \cup N_G(v)$. Since the induced subgraph of $H'$ on $N_G[\e']$ is the complete bipartite graph $K_{U',V}$, we have $N_{H'}(u) = V$ and $N_{H'}(v) = U$. Thus, $I(H')^2:e_t \supseteq I(K_{U,V})$. The conclusion follows from \cite[Theorem 6.7]{B} as any new edge of $H$ must have support in $N[\e]$.
\end{proof}
\begin{rem}
    The notion of bipartite completion of a bipartite subgraph of a simple graph $G$ was introduced and studied in \cite{MTV2}. It plays an important role in the study of depth of symbolic powers of $I(G)$.
\end{rem}

\subsection{Projective dimension of edge ideals of weakly chordal graphs}
We note that the colon ideals of powers of edge ideals of trees by products of edges are edge ideals of weakly chordal graphs. Their projective dimension can be computed via the notion of strongly disjoint families of complete bipartite subgraphs, introduced by Kimura \cite{K}. For a graph G, we consider all families of (non-induced) subgraphs $B_1, \ldots, B_g$ of $G$ such that
\begin{enumerate}
    \item each $B_i$ is a complete bipartite graph for $1 \le i \le g$,
    \item the graphs $B_1, \ldots, B_g$ have pairwise disjoint vertex sets,
    \item there exist an induced matching $e_1,\ldots, e_g$ of $G$ for each $e_i \in E(B_i)$ for $1\le i \le g$.
\end{enumerate}
Such a family is termed a strongly disjoint family of complete bipartite subgraphs. We define 
$$ d(G) = \max ( \sum_1^g |V (B_i)| - g),$$
where the maximum is taken over all the strongly disjoint families of complete bipartite subgraphs $B_1, \ldots, B_g$ of $G$. We have the following \cite[Theorem 7.7]{NV1}.

\begin{thm}\label{thm_pd_weakly_chordal}
    Let $G$ be a weakly chordal graph with at least one edge. Then 
    $$\pd (S/I(G)) = d(G).$$
\end{thm}

\section{Depth of powers of edge ideals of Cohen-Macaulay trees}\label{sec_depth_power_CM_trees}
In this section, we compute the depth of powers of Cohen-Macaulay trees. First, we prove a general bound for the depth of powers of edge ideals of graphs. Recall that for a set $\e = \{e_1, \ldots, e_t\}$ of edges of $G$, $G[\e]$ denotes the induced subgraph of $G$ on $N[\e]$ and $G[\bar {\e}]$ denotes the induced subgraph of $G$ on $V(G) \setminus N[\e]$.

\begin{lem}\label{lem_upperbound_power_tree} Let $\e = \{e_1, \ldots, e_t\}$ be a set of $t$ distinct edges of a simple graph $G$. Assume that $G[\e]$ is bipartite, $G[\bar {\e}]$ is weakly chordal, and $\e$, viewed as a subgraph of $G$, is connected. Let $R$ be the polynomial ring over the vertex set of $G[\bar{\e}]$. Then 
$$\depth S/I(G)^{t+1}  \le 1 + \depth R/I(G[\bar {\e}]).$$    
\end{lem}
\begin{proof}
    Let $K = I(G[\bar{\e}])$ and $f = e_1 \cdots e_t$. By Lemma \ref{lem_colon_product_edges}, $I(G)^{t+1} : f = I(H)$ where $E(H) = E(G) \cup E(\widetilde {G[\e]})$. If $G[\bar{\e}]$ has no edges then $K$ is the zero ideal and $\depth R/K = |V(G[\bar{\e}])|$. The conclusion then follows from \cite[Theorem 1.1]{K}. Now assume that $G[\bar{\e}]$ has at least one edge. Since $G[\bar{\e}]$ is weakly chordal, by Theorem \ref{thm_pd_weakly_chordal}, there exists a family $\B = \{B_1, \ldots, B_g\}$ of strongly disjoint family of complete bipartite subgraphs of $G[\bar{\e}]$ such that $\pd (R/K) = \sum_{i=1}^g |V(B_i)| - g$. Then $\B \cup \widetilde{G[\e]}$ is a strongly disjoint family of complete bipartite subgraphs of $H$, because $e_1 \in \widetilde{G[\e]}$ together with the induced matching $e_i' \in B_i$ form an induced matching of $H$. By \cite[Theorem 1.1]{K}, 
    $$\pd S/I(H) \ge \pd (R/K) + |V(\widetilde{G[\e]})| - 1.$$ 
    By the Auslander-Buchsbaum formula, we deduce that $\depth S/I(H) \le \depth R/K + 1$. The conclusion then follows from Lemma \ref{lem_upperbound}.
\end{proof}

As a corollary, we deduce an upper bound for the depth of powers of edge ideals of Cohen-Macaulay trees. By the result of Villarreal \cite{V}, we may assume that 
$$V(G) = \{x_1, \ldots, x_d,y_1, \ldots, y_d\} \text{ and }  E(G) = E(T) \cup \{\{x_1,y_1\},\ldots, \{x_d,y_d\}\},$$ where $T$ is the induced subgraph of $G$ on $\{x_1, \ldots, x_d\}$ which is a tree on $d$ vertices.

\begin{cor}\label{cor_upperbound_CM_tree} Let $G$ be a Cohen-Macaulay tree of dimension $d$. Then for any $t$ such that $2 \le t \le d-2$,
$$\depth S/I(G)^{t+1} \le d-t.$$    
\end{cor}
\begin{proof} Let $T_1$ be any connected subtree of $T$ with $|E(T_1)| = t$. We may assume that $V(T_1) = \{x_1,\ldots,x_{t+1}\}$ and
$$N_G[T_1] = \{x_1,\ldots,x_{t+1},\ldots,x_{t+1},\ldots,x_a,y_1,\ldots,y_{t+1}\}$$ 
for some $a \ge t+1$. Let $H_2$ be the induced subgraph of $G$ on $V(G) \setminus N_G[V(T_1)] = \{ x_{a+1},\ldots,x_d,y_{t+2},\ldots,y_d\}$. Then $H_2$ is the whisker graph of $T_2$, the induced subgraph of $T$ on $\{x_{a+1},\ldots,x_d\}$ and $\{y_{t+2},\ldots,y_{a}\}$ are isolated vertices of $H_2$. By \cite{V}, $I(H_2)$ is Cohen-Macaulay of dimension $d - t - 1$. By Lemma \ref{lem_upperbound_power_tree}, we have 
$$\depth S/I(G)^{t+1} \le 1 + \depth R/I(H_2) = d-t.$$
The conclusion follows.
\end{proof}

We now prove the following lower bound for the depth of powers of edge ideals of general whisker trees.

\begin{lem}\label{lem_lower_bound_whisker_trees} Let $T$ be a forest on $n$ vertices. Let $U \subset V(T)$ be a subset of vertices of $T$, and $G$ be the new forest obtained by adding a whisker to each vertex in $U$. Namely, $V(G) = V(T) \cup \{y_i \mid x_i \in U\}$ and $E(G) = E(T) \cup \{ \{x_i, y_i\} \mid x_i \in U\}$. Then for all $t \ge 1$, we have 
$$\depth (R/I(G)^t) \ge |U| - t +  1,$$
where $R = \k[x_1,\ldots,x_n,y_j \mid x_j \in U]$ is the polynomial ring over the variables corresponding to vertices of $G$.
\end{lem}
\begin{proof} By \cite[Theorem 1.1]{NV2}, we may assume that $T$ is a tree on $n$ vertices. Furthermore, we may assume that $U = \{1, \ldots, d\}$. Hence, $V(G) = \{x_1, \ldots, x_n, y_1, \ldots, y_d\}$ and $E(G) = E(T) \cup \{x_i y_i \mid i = 1, \ldots, d\}$. Since $G$ is a forest, $\depth R/I(G)^t \ge 1$ for all $t \ge 1$. Hence, the statement is vacuous when $t \ge d$. Thus, we may assume that $1 \le t \le d$. We prove by induction on the triple $(n,d,t)$ ordered by the lexicographic order the following 
\begin{equation}\label{eq_3_1}
    \depth R/ (I(G)^t + I(H)) \ge d - t + 1,
\end{equation}
for all $t \le d$ and all subgraphs $H$ of $G$ with $E(H) \subseteq \{x_1y_1, \ldots, x_dy_d\}$. 

For ease of reading, we divide the proof into several steps. 

\vspace{1mm}
\noindent {\bf Step 1.} The base case $d = 1$ is clear as $\m$, the maximal homogeneous ideal of $R$ is not an associated prime of $I(G)^t + I(H)$.

\vspace{1mm}

\noindent {\bf Step 2.} The base case $t = 1$. The statement follows from \cite[Theorem 4.1]{MTV1} as a maximal independent set of $G$ must contain either $x_i$ or $y_i$ for each $i = 1, \ldots, d$. 

\vspace{1mm}

\noindent {\bf Step 3.} Reduction to the case $E(H) = \{x_1y_1,\ldots,x_dy_d\}$. Assume that $E(H)$ is a proper subset of $\{x_1y_1, \ldots, x_dy_d\}$, say $x_dy_d \notin E(H)$. Let $J = I(G)^t + I(H)$. By Lemma \ref{lem_depth_colon}, Lemma \ref{lem_colon_leaf}, and induction, it suffices to prove that $\depth R/(J + (x_dy_d)) \ge d- t + 1$. Thus, we may assume that $E(H) = \{x_1y_1, \ldots, x_dy_d\}$. Eq. \eqref{eq_3_1} becomes
\begin{equation}\label{eq_3_2}
    \depth (R/I(T)^t + (x_1y_1,\ldots,x_dy_d)) \ge d - t + 1.
\end{equation}

\vspace{1mm}

\noindent {\bf Step 4.} Induction step. Assume that Eq. \eqref{eq_3_2} holds for all tuples $(n',d',t')$ strictly smaller than $(n,d,t)$ in the lexicographic order. Let $J = I(T)^t + (x_1y_1, \ldots, x_dy_d)$. Let $u$ be a leaf of $T$ and $v$ the unique neighbour of $u$ in $T$. There are two cases.

\vspace{1mm}

{\bf Case 1.} $u \notin \{x_1, \ldots, x_d\}$. Since $J + (u)$ is of the same form but in a smaller ring, by induction, we have $\depth R/ (J + ( u)) \ge d - t + 1$. By Lemma \ref{lem_depth_colon}, it suffices to prove that 
$$\depth R/(J:u) \ge d - t + 1.$$
But $K = J:u = v I(T)^{t-1} + (x_1y_1,\ldots,x_dy_d)$. Since $\depth R/ (K + (v)) \ge d$, by Lemma \ref{lem_depth_colon}, it suffices to prove that $\depth R/(K:v) \ge d - t + 1$. There are two subcases.

Subcase 1.a. $v \notin \{x_1, \ldots, x_d\}$. Then $K:v = I^{t-1} + (x_1y_1,\ldots, x_dy_d)$. Hence, by induction on $t$, we have $\depth R/(K:v) \ge d - (t-1) + 1$.

Subcase 1.b. $v \in \{x_1, \ldots, x_d\}$, say $v = x_d$. Then 
$$K:v = I^{t-1} + (x_1y_1, \ldots, x_{d-1}y_{d-1}) + (y_d).$$ 
Hence, by induction, we have $\depth R/(K:v) \ge (d-1) - (t-1) + 1 = d-t+1$.

\vspace{1mm}

{\bf Case 2.} $u \in \{x_1, \ldots, x_d\}$, say $u = x_d$. Let $T_1$ be the subtree of $T$ restricted to $V(T) \setminus \{u\}$. Then $J + (u) = I(T_1)^t + (x_1y_1,\ldots, x_{d-1}y_{d-1}) + (u)$. In particular, $y_d$ is a free variable. Hence, $\depth R/(J + (u)) = 1 + \depth R_1 / (I(T_1)^t + (x_1y_1,\ldots,x_{d-1}y_{d-1})),$ where $R_1 = \k[x_i,y_j \mid i,j \neq d]$. By induction, we have $\depth R/ (J + (u)) \ge 1 + (d-1) - t + 1 = d - t + 1$. By Lemma \ref{lem_depth_colon}, it suffices to prove that $\depth R/(J:u) \ge d - t + 1$. 

We have 
$$J:u = v I(T)^{t-1} + (x_1y_1,\ldots,x_{d-1}y_{d-1}) + (y_d).$$ 

Let $R_1 = \k[x_1,\ldots, x_n,y_1,\ldots,y_{d-1}]$ and $K = v I(T)^{t-1} + (x_1y_1,\ldots,x_{d-1}y_{d-1})$. Then, $\depth R/K = \depth R_1/K$. By Lemma \ref{lem_depth_colon}, it suffices to prove that $\depth R_1/K:v \ge d - t + 1$.

Again, we have two subcases,

Subcase 2.a. $v \notin \{x_1, \ldots, x_{d-1}\}$. Then $K:v = I(T)^{t-1} + (x_1y_1,\ldots,x_{d-1}y_{d-1})$. Hence, by induction, $\depth R_1/(K:v) \ge (d-1) - (t-1) + 1$. 

Subcase 2.b. $v \in \{x_1, \ldots, x_{d-1}\}$, say $v = x_{d-1}$. Then, $K:v = I(T)^{t-1} + (x_1y_1,\ldots, x_{d-2}y_{d-2}) + (y_{d-1}).$ Let $T'$ be the induced subgraph of $T$ on $V(T) \setminus \{d\}$. If $t =2$ then 
$$K:v = I(T') + (x_1y_1,\ldots,x_{d-2}y_{d-2},x_{d-1}x_d) + (y_{d-1}).$$ 
The conclusion follows from \cite[Theorem 4.1]{MTV1}. Now, assume that $t \ge 3$. Let $R_2 = \k[x_1,\ldots,x_n,y_1,\ldots,y_{d-2}]$ and $L = I(T)^{t-1} + (x_1y_1,\ldots,x_{d-2}y_{d-2})$. We need to prove that $\depth R_2/L \ge d - t + 1$. Note that $I(T) = I(T') + (x_{d-1}x_d)$ and $x_d$ can be considered as a whisker at $x_{d-1}$. By Lemma \ref{lem_depth_colon}, we have 
$$\depth R_2/L \in \{\depth (R_2/L:(x_{d-1}x_d)), \depth (R_2/ L + (x_{d-1}x_d))\}.$$
By Lemma \ref{lem_colon_leaf}, $L:x_{d-1}x_d = I(T)^{t-2} + (x_1y_1,\ldots,x_{d-2}y_{d-2})$. Hence, by induction, we have $\depth R_2/(L:x_{d-1}x_d) \ge (d-2) - (t-2) + 1 = d- t + 1$. Finally, we have $L + (x_{d-1}x_d) = I(T')^{t-1} + (x_1y_1,\ldots,x_{d-2}y_{d-2},x_{d-1}x_d)$, with $x_d$ now play the role of $y_{d-1}$. By induction, we also have $\depth R_2 / (L + (x_{d-1}x_d)) \ge d-1 - (t-1) + 1 = d-t + 1$.

That concludes the proof of the Lemma.
\end{proof}

We are ready for the main result of this section.
\begin{thm}Let $I(G) \subset S = \k[x_1, \ldots,x_{d}, y_1, \ldots, y_d]$ be the edge ideal of a Cohen-Macaulay tree $G$ of dimension $d$. Then for all $t \ge 1$,
    $$\depth S/I(G)^t =  \max \{d-t+1, 1\}.$$    
\end{thm}
\begin{proof}
    The conclusion follows from the result of Villarreal \cite{V}, Corollary \ref{cor_upperbound_CM_tree}, and Lemma \ref{lem_lower_bound_whisker_trees}.
\end{proof}

\section{Depth of powers of edge ideals of some Cohen-Macaulay bipartite graphs}\label{sec_depth_power_CM_bipartite_graph}
In this section, we study the depth of powers of some Cohen-Macaulay bipartite graphs. First, we consider a graph constructed by Banerjee and Munkudan \cite{BM}. 

\begin{thm}\label{thm_depth_BM_ideal} Let $S = k[x_1,\ldots,x_d,y_1,\ldots,y_d]$. For a fix integer $j$ such that $1 \le j \le d$, let $G_{d,j}$ be a bipartite graph whose edge ideal is $I(G_{d,j}) = (x_i y_i, x_1 y_i, x_k y_j \mid 1 \le i \le d, 1 \le k \le j)$. Then 
$$\depth S/I(G_{d,j})^s = \max(1,d - j -s + 3),$$
for all $s \ge 2$.    
\end{thm}
\begin{proof} Fix $j$, we prove by induction on $d$. The case $d = j$ follows from \cite[Example 4.2]{BM}. Now assume that the statement holds for $d-1$. Let $s$ be an exponent such that $2 \le s \le d - j - 2$. Let $e_1 = x_1y_j, e_2 = x_1y_{j+1}, \ldots, e_{s-1} = x_1 y_{j + s-2}$ and $\e = \{e_1, \ldots, e_{s-1}\}$. We have $G_{d,j}[\bar {\e}]$ is the disjoint union on $d-j-s+2$ edges. By Lemma \ref{lem_upperbound_power_tree}, 
$$\depth S/I(G_{d,j})^s \le 1 + \depth R/I(G_{d,j}[\bar{\e}]) = d- j - s + 3.$$
In particular, $\depth S/I(G_{d,j})^{d-j-2} \le 1$. By \cite[Theorem 4.4]{T}, we deduce that $\depth S/I(G_{d,j})^t = 1$ for all $t \ge d-j-2$.

For the lower bound, the proof is similar to that of Lemma \ref{lem_lower_bound_whisker_trees}. Let $H$ be the induced subgraph of $G_{d,j}$ on $\{x_1,\ldots,x_j,y_1,\ldots,y_d\}$. Then 
\begin{align*}
    I(H) &= (x_1y_i,x_ky_k,x_ky_j \mid 1 \le i \le d, 2 \le k \le j)\\
    I(G_{d,j}) &= I(H) + (x_iy_i \mid i \ge j + 1).
\end{align*}
    
For ease of reading, we divide the proof into several steps. 

\vspace{1mm}

\noindent {\bf Step 1.} With an argument similar to Step 3 of the proof of Lemma \ref{lem_lower_bound_whisker_trees}, we reduce to proving the following lower bound
\begin{equation}
    \depth S/(I(H)^s + (x_iy_i \mid i \ge j + 1)) \ge d-j-s + 3.
\end{equation}

\vspace{1mm}
\noindent {\bf Step 2.} Let $J = I(H)^s + (x_iy_i \mid i \ge j + 1)$. By Lemma \ref{lem_depth_colon}, we will prove the bound for $\depth S/(J + (y_d))$ and $\depth S/(J:y_d)$. We have 
$$J + (y_d) = I(H)^s + (x_1y_1,x_iy_i \mid j+1 \le i \le d-1) + (y_d).$$ 
Hence, $x_d$ is a free variable of $J + (y_d)$. By induction, we have 
$$\depth S/(J + (y_d)) \ge 1 + (d-1 - j +s + 3) = d - j -s +3.$$
Furthermore, 
$$J:y_d = x_1 I(H)^{s-1} + (x_iy_i \mid j+1 \le i \le d-1) +  (x_d).$$
Let $K = x_1 I(H)^{s-1} + (x_iy_i \mid j+1 \le i \le d-1)$ and $R = \k[x_1,\ldots,x_{d-1},y_1,\ldots,y_{d-1}]$. Then, $\depth S/(J:y_d) = \depth R/K$. Since $K+(x_1) = (x_iy_i \mid j+1 \le i \le d-1) + (x_1)$, by Lemma \ref{lem_depth_colon}, it remains to bound $\depth R/K:x_1$. We have
$$K:x_1 = I(H)^{s-1} + (x_iy_i \mid j+1 \le i \le d-1).$$
Hence, by induction,
$$\depth R/(K:x_1) \ge d-1 - (s-1) - j + 3 = d -j -s + 3.$$
That concludes the proof of the Theorem.
\end{proof}

In \cite{BM}, Banerjee and Mukundan said that the depth sequence of powers of the edge ideal of a graph $G$ has a drop at $k$ if $\depth S/I^k - \depth S/I^{k+1} > 1$. They then constructed a Cohen-Macaulay bipartite graph with an arbitrary number of drops in their depth sequence. Nonetheless, their construction is via sum of ideals, and hence, the resulting Cohen-Macaulay bipartite graph $G$ with $k$ drops has $k$ connected components. By the result of Nguyen and the last author \cite{NV2}, we may reduce the computation of depth of powers of edge ideals of disconnected graphs to that of their connected components. By our computation experiment, we believe that we may construct a connected Cohen-Macaulay bipartite graph with an arbitrary number of drops. To conclude, we provide an example of a connected Cohen-Macaulay bipartite graph with two drops in its depth sequence of powers. 

\begin{thm}\label{exm_CM_bipartite_graphs_two_drops} Assume that $d \ge 5$. Let $S = k[x_1,\ldots,x_d,y_1,\ldots,y_d]$. For each $a$ such that $2 \le a \le d-1$, let $G_{d,a}$ be a bipartite graph whose edge ideal is 
$$I(G_{d,a}) = x_1(y_1, \ldots, y_d) + (x_iy_j \mid 2 \le i \le j \le a) + (x_iy_j \mid a+1 \le i \le j \le d).$$
Then
$$\depth S/I(G_{d,a})^t = \begin{cases} 
d & \text{ if } t = 1\\
\min(a,d-a+1) & \text{ if } t = 2\\
1 & \text{ if } t \ge 3.\end{cases}$$       
\end{thm}
\begin{proof} Note that $G_{d,a}$ is a Cohen-Macaulay bipartite graph of dimension $d$ by \cite{EV, HH}. By Lemma \ref{lem_colon_product_edges},
$$I(G_{d,a})^3 : (x_1 y_a x_1 y_{d}) = I(K_{d,d}).$$
By Lemma \ref{lem_upperbound} and Lemma \ref{lem_depth_power_non_increasing}, $\depth S/I(G_{d,a})^t =1$ for all $t \ge 3$. Thus, it remains to determine $\depth S/I(G_{d,a})^2$. When $\e = \{x_1y_a\}$, $G_{d,a}[\bar{\e}]$ is the empty graph on $d-a$ vertices $\{x_{a+1},\ldots, x_d\}$. When $\e = \{x_1y_d\}$, $G_{d,a}[\bar{\e}]$ is the empty graph on $a-1$ vertices $\{x_2,\ldots,x_a\}$. Hence, by Lemma \ref{lem_upperbound_power_tree},
$$\depth S/I(G_{d,a})^2 \le \min(a,d-a+1).$$

For the lower bound, we prove by induction on the tuple $(d,a)$. We divide the proof into several steps.

\vspace{1mm}

\noindent {\bf Step 1.} $\depth S/(I(G_{d,a})^2 + (y_d)) \ge \min (a,d-a+1)$. Since $I(G_{d,a})^2 + (y_d) = I(G_{d-1,a})^2 + (y_d)$ and $x_d$ is a free variable of $I(G_{d,a})^2 + (y_d)$. By induction, we have 
$$\depth S / (I(G_{d,a})^2 + (y_d)) \ge 1 + \min (a,d-1-a+1) \ge \min(a,d-a+1).$$

\vspace{1mm}

\noindent {\bf Step 2.} $\depth S/(I(G_{d,a})^2 + (x_1)) \ge \min(a,d-a+1)$. Since $I(G_{d,a})^2 + (x_1) = (I(H_1) + I(H_2))^2 + (x_1)$, where 
\begin{align*}
    I(H_1) &= (x_iy_j \mid 2 \le i \le j \le a),\\
    I(H_2) &= (x_iy_j \mid a+1 \le i \le j \le d)
\end{align*}
are Cohen-Macaulay ideals of dimensions $a-1$ and $d-a$ respectively. Since $y_1$ is a free variable of $I(G_{d,a})^2 + (x_1)$, by \cite[Theorem 1.1]{NV2},
\begin{align*}
    \depth S/(I(G_{d,a})^2 + (x_1))  = 1 + \min (& \depth R_1 / I(H_1) + \depth R_2 / I(H_2) + 1,\\
    &\depth R_1/I(H_1)^2 + \depth R/I(H_2),\\
    &\depth R_1/I(H_1) + \depth R_2 / I(H_2)^2) \\
    &= 1 + \min (a,d-a+1),
\end{align*}
where $R_1=\k[x_i,y_i \mid i = 2, \ldots, a]$ and $R_2 = \k[x_i,y_i \mid i = a+1,\ldots, d]$.

\vspace{1mm}

\noindent {\bf Step 3.} $\depth S/(I(G_{d,a})^2 + (x_1,y_d)) \ge \min(a,d-a+1)$. Since $I(G_{d,a}) + (x_1,y_d)$ is the mixed sum of two Cohen-Macaulay ideals of dimensions $a-1$ and $d-a-1$ respectively with free variables $y_1,x_d$ in $S$. With argument similar to Step 2, we deduce the desired lower bound.

\vspace{1mm}

\noindent {\bf Step 4.} $\depth S/(I(G_{d,a})^2 + x_1y_d) \ge \min (a,d-a+1)$. Since $I(G_{d,a})^2 + (x_1y_d) = (I(G_{d,a})^2 + (x_1)) \cap (I(G_{d,a})^2 + (y_d))$. The conclusion follows from Step 1, Step 2, Step 3, and a standard lemma on depth \cite[Proposition 1.2.9]{BH}.

\vspace{1mm}

\noindent {\bf Step 5.} $\depth S/(I(G_{d,a})^2:(x_1y_d)) \ge \min (a,d-a+1)$. By Lemma \ref{lem_colon_product_edges},
\begin{equation}\label{eq_4_3}
    I(G_{d,a})^2 : (x_1y_d) = I(G_{d,a}) + (x_{a+1},\ldots,x_d) \cdot (y_1,\ldots,y_d).
\end{equation}
Let $L = I(G_{d,a})^2 : (x_1y_d)$. Then 
$$L = (I(G_{d,a}) + (x_{a+1},\ldots,x_d)) \cap (y_1,\ldots,y_d).$$ 
Since $\depth S/(y_1,\ldots,y_d) = d$ and $\depth S/(x_{a+1},\ldots,x_d,y_1,\ldots,y_d) = a$, it remains to show that 
\begin{equation}
    \depth S/M \ge \min(a,d-a+1),
\end{equation}
where $M = I(G_{d,a}) + (x_{a+1},\ldots,x_d)$. We have $M:x_1 = (x_{a+1},\ldots,x_d,y_1,\ldots,y_d)$ and $M+(x_1)$ is a Cohen-Macaulay ideal of dimension $a-1$ and $y_1,y_{a+1},\ldots,y_d$ are free variables of $M + (x_1)$. By Lemma \ref{lem_depth_colon}, $\depth S/M \ge \min(a,d-a+1)$ as required.

Thus, we deduce that $\depth S/I(G_{d,a})^2 \ge \min(a,d-a+1)$ by Step 4, Step 5 and Lemma \ref{lem_depth_colon}. That concludes the proof of the Theorem.
\end{proof}

\section*{Acknowledgments}
Nguyen Thu Hang is partially supported by the Thai Nguyen University of Sciences (TNUS) under the grant number CS2021-TN06-16.

\end{document}